\documentclass[a4paper,11pt]{amsart}

\usepackage{style}

\title[Rational curves on primitive symplectic varieties of OG\(_6^s\)-type]{Rational curves on primitive symplectic varieties of OG\(_6^s\)-type}

\author{Valeria Bertini}
\address{Mathematics Center of the University of Porto, Portugal}
\email{valeriabertini90@gmail.com}

\author{Annalisa Grossi}
\address{Fakult\"at f\"ur Mathematik - TU-Chemnitz, Deutschland}
\email{annalisa.grossi@math.tu-chemnitz.de}

\makeatletter
\@namedef{subjclassname@2020}{%
	\textup{2020} Mathematics Subject Classification}
\makeatother



\thanks{}

\begin{document}
	
	\maketitle
	
	\begin{abstract}
	We prove that any ample class on a primitive symplectic variety that is locally trivial deformation of O'Grady's singular 6-dimensional example is proportional to the first Chern class of a uniruled divisor. This result answers a question of Lehn--Mongardi--Pacienza \cite[Remark 4.7]{LMP}, extending their result \cite[Theorem 1.3]{LMP} for primitive symplectic varieties of this deformation type.
		
	\end{abstract}

\tableofcontents
\addtocontents{toc}{\protect\setcounter{tocdepth}{1}}

\section{Introduction}\label{section introduction}
\subsection{Background}\label{subsection background}
An \textit{irreducible holomorphic symplectic manifold}, in short \textit{ihs manifold}, is a compact complex K\"ahler manifold \(X\) that is simply connected and such that \(\HH^0(X,\Omega_X^{ 2})\) is generated by an everywhere non-degenerate holomorphic \(2\)-form, called the symplectic form of \(X\).
Ihs manifolds of dimension \(2\) are \(\K3\) surfaces, and Beauville \cite{beauville.varietes} exhibited two examples in each even complex dimension: Hilbert schemes of zero dimensional subschemes of length \(n\) on a \(\K3\) surface and generalized Kummer manifolds, whose deformation types are called \(\K3^{[n]}\)-type and \(\K_n(A)\)-type respectively.
O'Grady \cite{OG6,OG10} gave two further examples in dimension six and ten. Ihs manifolds deformation equivalent to the O'Grady's examples are called of \(\OG_6\)-type and \(\OG_{10}\)-type respectively. 

In recent years, \textit{singular} symplectic varieties have been object of great interest and their theory has been very much explored and developed; we refer to  \autoref{subsection psv and O'Grady's ex} for the definitions in the singular setting. As for the smooth case, singular symplectic varieties appear as factors of normal K\"ahler spaces with mild singularties and trivial first Chern class, giving  a singular version of the Beauville-Bogomolov decomposition theorem \cite{greb2016singular} \cite{horing2019algebraic} \cite{BGL22}. Ways to produce singular symplectic varieties are to consider quotients of smooth ihs manifolds by a finite group of symplectic automorphisms \cite[Proposition 2.4]{beauville.symplectic}, or to consider singular moduli spaces of sheaves on surfaces with trivial canonical bundle \cite[Theorem 1.19]{PR18}.

The class of singular symplectic varieties that we consider throughout the paper are \textit{primitive symplectic varieties}, see \autoref{def sing sympl}. The notion of primitive symplectic variety generalizes the definition of ihs manifold to the singular setting, and in fact  \textit{smooth} primitive symplectic variety do coincide with  ihs manifold, \cite[Lemma 3.3]{BL18} and \cite[Theorem 1]{schwald20}. The theory of primitive symplectic varieties behaves similarly to the one of ihs manifolds in many relevant aspects, as the lattice structure on second integral cohomology group (see \autoref{thm BBF and Fujiki}), their deformation theory and Torelli type results \cite{BL18}.

\subsection{Rational curves on primitive symplectic varieties} In this paper we focus on another interesting property shared by ihs manifolds and primitive symplectic varieties: the existence of rational curves on them, and more precisely, the existence of ample uniruled divisors (see \autoref{def uniruled}).  This problem has been firstly investigated in the smooth setting, where the existence of ample uniruled divisors on an ihs manifold allows to build a canonical subgroup of the 0-Chow group of the ihs manifold \cite[Theorem 1.5]{CPM}, giving a first evidence of Voisin's geometrical realization of the conjectural Bloch-Beilinson filtration of the 0-Chow group \cite{voisin.remarks}. Uniruled divisors are known to exist in almost any ample linear system on ihs manifolds of \(\K3^{[n]}\)-type \cite{CPM} and of \(\K_n(A)\)-type \cite{mongardi.pacienza} \cite{mongardi2019corrigendum}, and in any ample linear system with some fixed numerical invariants for ihs manifolds of \(\OG_{10}\)-type \cite{bertini2021rational}. The key ingredient in all these cases is a result by Charles-Mongardi-Pacienza \cite{CPM} about deformations of rational curves ruling a divisor on an ihs manifold. This result allows to pass from the construction of some explicit examples of ample uniruled divisors to the existence of uniruled divisors in ample classes which are deformation of the ones explicitly constructed.

\subsection{Main results}\label{subsection motivation and results}

Lehn-Mongardi-Pacienza recently extended the Charles-Mongardi-Pacienza  deformation result to the singular setting, i.e. for primitive symplectic varieties \cite[Theorem 1.1]{LMP}, and they started the search of ample uniruled divisors on them. In this paper we focus on primitive symplectic varieties that are locally trivial deformation of the singular 6-dimensional example of O'Grady, that we call of \(\OG_6^s\)-type, see \autoref{def OG6 K_v}.
We prove that any ample class on a primitive symplectic variety of \(\OG_6^s\)-type is proportional to the first Chern class of a uniruled divisor.

\begin{theorem}\label{main thm}
Let \((X,h)\) be a polarized primitive symplectic variety of \(\OG_6^s\)-type. Then there exists a positive integer \(m\) such that \(mh\) is the first Chern class of a uniruled divisor. 
\end{theorem}

Let \(\mathfrak M_{\OG_6^s}\) be the moduli space of polarized primitive symplectic varieties of \(\OG_6^s\)-type, see \autoref{def mod space of polarized psv}. Lehn-Mongardi-Pacienza prove in \cite[Theorem 1.3]{LMP}  that \(\mathfrak M_{\OG_6^s}\) has infinitely many connected components whose points have polarization that is proportional to the first Chern class of a uniruled divisor, but they can not give any characterization of these components. With \autoref{main thm} we conclude that every connected component of \(\mathfrak M_{\OG_6^s}\) has this property, answering their question \cite[Remark 4.7]{LMP}.

\autoref{main thm} is based on the existence of positive effective irreducible divisors on some specific \(\OG_6^s\)-type primitive symplectic varieties, called \(\OG_6^s\) moduli spaces, as stated in \autoref{thm_positive}. The conclusion follows by deforming the rational curves ruling these positive effective irreducible divisors via \cite[Theorem 1.3]{LMP}, and any polarized primitive symplectic variety of \(\OG_6^s\)-type is reached in this way: two of such varieties are one locally trivial polarized deformation of the other exactly when their polarizations have the same square and the same divisibility (see \autoref{subsection monodromy_OG6_eichler}), and all possible squares and divisibilities arise on \(\OG_6^s\) moduli spaces by \autoref{prop invariants of classes on Kv}.

\subsection{Structure of the paper} This paper is organized as follows. In \autoref{section preliminaries} we introduce the notion of primitive symplectic varieties and the \(\OG^s_6\) deformation class, together with few fundamental results about them. In \autoref{subsection MRS model} we recall the definition of a certain regular morphism from a \(\K3^{[3]}\)-type ihs manifold to a \(\OG_6^s\)-type primitive symplectic variety introduced by Mongardi--Rapagnetta--Saccà, that we use in \autoref{subsection positive uniruled on K} to conclude the existence of uniruled divisors in any positive effective class on a \(\OG_6^s\) moduli space, see \autoref{thm_positive}. In \autoref{subsection positive uniruled on K} we introduce the notion of polarized moduli space of \(\OG_6^s\)-type primitive symplectic varieties \(\mathfrak M_{\OG_6^s}\). In \autoref{section.oribts} we pass to any polarized \(\OG_6^s\)-type primitive symplectic variety in \(\mathfrak M_{\OG_6^s}\): as first in \autoref{subsection monodromy_OG6_eichler} we recall that the connected components of \(\mathfrak M_{\OG^s_6}\) are characterized by degree and divisibility of the polarization of an element, then in \autoref{subsection main results} we exhibit a positive effective class on a \(\OG_6^s\) moduli space for all possible squares and divisibilities, see \autoref{lemma_ type_d} and \autoref{prop invariants of classes on Kv}, and we prove \autoref{main thm}. We devote \autoref{section smooth case} to some considerations on the smooth case, i.e. ihs manifolds of \(\OG_6\)-type.


\subsection*{Acknowledgements} 
We wish to thank Christian Lehn, Giovanni Mongardi, Claudio Onorati and Gianluca Pacienza for their hints and suggestions.

\section{Preliminaries}\label{section preliminaries} 

\subsection{Primitive symplectic varieties and O'Grady's example}\label{subsection psv and O'Grady's ex} We start recalling the basics definitions in the smooth setting.

\begin{definition}[Holomorphic symplectic manifolds - smooth setting] Let \(X\) be a complex manifold.
\begin{enumerate}
    \item A form \(\sigma\in \HH^0(X,\Omega_X^2)\) is called \textit{holomorphic symplectic} if it is everywhere non-degenerate on \(X\).
    \item \(X\) is a \textit{holomorphic symplectic manifold} if there exists a holomorphic symplectic form on \(X\).
    \item \(X\) is an \textit{irreducible holomorphic symplectic} (in short, \textit{ihs}) manifold if it is a simply connected compact K\"ahler holomorphic symplectic manifold such that \(h^0(X,\Omega^2_X)=1\).
\end{enumerate}
\end{definition}

In the singular setting there are several notions of symplectic varieties, whose nomenclature is not uniform in literature; for our definitions we refer to \cite[\S 3]{BL18} and \cite[\S 1.1, \S 1.2]{PR18}. Let \(X\) be a normal variety and \(j:X^{sm}\hookrightarrow X\)  the embedding of its smooth locus. For any integer \(0\le p\le \dim(X)\) we define 
\[
\Omega_X^{[p]}:=j_*\Omega^p_{X^{sm}}=( \Omega^p_X)^{\vee\vee}
\]
to be the sheaf of reflexive \(p\)-forms on \(X\); here \(\Omega_X^\vee=\Hom_{\mathcal O_X}(\Omega^p_X,\mathcal O_X)\) is the dual sheaf as sheaf of \(\mathcal O_X\)-modules. 

\begin{definition}\label{def sing sympl}[Symplectic varieties - singular setting] Let \(X\) be a normal variety.
\begin{enumerate}
    \item A reflexive form \(\sigma\in \HH^0(X,\Omega_X^{[2]})\) is called \textit{symplectic} if its restriction \(\sigma|_{X^{sm}}\) is a holomorphic symplectic form on \(X^{sm}\).
    \item A pair \((X,\sigma)\) is a \textit{symplectic variety} if \(\sigma\) is a symplectic form on \(X\) and for some (hence any) resolution of singularities \(f:\widetilde X\to X\) the pullback \(f^*\sigma|_{X^{sm}}\) extends holomorphically to \(\widetilde X\).
    \item Given a symplectic variety \((X,\sigma)\) a resolution of singularities \(f:\widetilde X\to X\) is called a \textit{symplectic resolution} if \(f^*\sigma|_{X^{sm}}\) extends to a holomorphic symplectic form on \(\widetilde X\).
    \item A compact K\"ahler symplectic variety \((X,\sigma)\) is called a \textit{primitive symplectic variety} if \(h^1(X,\mathcal O_X)=0\) and \(h^0(X,\Omega_X^{[2]})=1\). If furthermore \(X\) is projective, then \((X,\sigma)\) is called a \textit{Namikawa symplectic variety}.
\end{enumerate}
When we don't need to keep track of the symplectic form \(\sigma\) we drop it from the notation.
\end{definition}

A particular class of primitive symplectic varieties are the irreducible-holomorphic symplectic (often called simply irreducible symplectic) varieties, introduced by \cite[Definition 8.16]{greb2016singular}; they are particularly relevant as they arise as one of the three building blocks in the Beauville-Bogomolov decomposition theorem in the singular setting we have recalled in \autoref{subsection background}. We will not be interested in this class of primitive varieties, so we omit their definition. 

By \cite[Proposition 1.3]{beauville.symplectic} a  symplectic variety \(X\) has rational singularities, i.e. it is normal and for any resolution of singularities \(f:\widetilde X\to X\) it holds \(R^if_*\mathcal O_{\widetilde X}=0\) for any \(i\ge 1\). As consequence one deduces the following.

\begin{proposition}\cite[Proposition 1.9]{PR18}\label{prop primitive sympl}
Let \(X\) be a connected projective symplectic variety
admitting a symplectic resolution \(f:\widetilde X\to X\) such that \(\widetilde X\) is an ihs manifold. Then  \(X \) is a Namikawa symplectic variety.
\end{proposition}

We denote by \(\HH^2(X,\ZZ)_{tf}\) the torsion-free part of the abelian group \(\HH^2(X,\ZZ)\). Another consequence of the fact that a symplectic variety has rational singularities is the following. 

\begin{proposition}\cite[Lemma 2.1, Corollary 3.5]{BL18} \label{prop cohomology sympl resol}
Let \(X\) be a compact symplectic variety. Then \(\HH^2(X,\ZZ)_{tf}\) carries a pure weight 2 Hodge structure. If furthermore \(f:\widetilde X\to X\) is a symplectic resolution of \(X\), then the pullback \(f^*:\HH^2(X,\mathbb Z)\to \HH^2(\widetilde X,\mathbb Z)\) is an inclusion of pure Hodge structures; in particular, \(\HH^2(X,\ZZ)\) is torsion-free.
\end{proposition}

The examples of (projective) ihs manifolds and primitive symplectic varieties we are going to work with are defined as moduli spaces of sheaves on canonical trivial surfaces. We are going to briefly recall here their definition and some results about them; for further details we refer to \cite[\S 1.3]{PR18}.

Let \(S\) be a projective \(\K3\) surface or an abelian surface. Given a coherent sheaf \(\mathcal F\in\Coh(S)\) the Mukai vector of \(\mathcal F\) is defined as
\[
\mathfrak v(\mathcal F):=\ch(\mathcal F)\sqrt{\td(S)}\in \HH^{2*}(S,\ZZ):=\HH^0(S,\ZZ)\oplus \HH^2(S,\ZZ)\oplus \HH^4(S,\ZZ).
\]
For any \(v=(r,l,s)\in \HH^{2*}(S,\ZZ)\) we say that \(v\) is a \textit{Mukai vector} if \(r\ge 0\) and \(l\in \NS(S)\), and if \(r=0\) then \(l\) is the first Chern class of an effective divisor or \(l=0\) and \(r>0\). Given a polarization \(h\) on \(S\), i.e. \(h=c_1(H)\) is the first Chern class of an ample divisor \(H\) on \(S\), there is a notion of \(v\)-genericity for \(h\), that we are not going to recall here; when \(\rk\Pic(S)=1\) then any polarization is \(v\)-generic, and this is the only case we are going to consider in the following sections.

\begin{definition}
Let \(S\) be a projective \(\K3\) or abelian surface, \(v\) a Mukai vector and \(h\) a \(v\)-generic polarization.
\begin{enumerate}
    \item We call \(\M_v(S,h)\) the moduli space of S-equivalence classes of Gieseker \(h\)-semistable sheaves \(\mathcal F\in \Coh(S)\) with Mukai vector \(\mathfrak v(\mathcal F)=v\).
    \item  When \(S\) is an abelian surface we call \(\K_v(S,h):=a_v^{-1}(0_S,\cO_{S})\subset\M_v(S,h)\), where \(a_v\) is the isotrivial fibration
\begin{align*}
    a_v:\M_v(S,h)&\to S\times S^\vee \\
    \mathcal F&\mapsto\bigl(Alb (c_2(\mathcal F)), \det(\mathcal F)\otimes\det(\mathcal F_0)^{-1}\bigl) 
\end{align*}
Here  \(Alb:\CH_0(S)\to S\) is the Albanese morphism and \([\mathcal F_0]\in \M_v(S,h)\) is the S-equivalence class of a fixed coherent sheaf.
\end{enumerate}
When no confusion on \(S\) and \(h\) is possible, we drop them from the notation.
\end{definition}

The moduli spaces \(\M_v\) and \(\K_v\) are projective, connected, normal varieties, and results on them are known by the work of several authors. We collect in the following theorem results about the only case we will be interested in in what follows.

\begin{theorem}\label{thm.K_v} Let \(A\) be an abelian surface, \(v\) a Mukai vector and \(h\) a \(v\)-generic polarization.
\begin{enumerate} 
    \item \cite{OG6} If \(v=(2,0,-2)\in \HH^{2*}(A,\ZZ)\) then \(\K_v=:\K_6\) is singular and it admits a symplectic resolution \(\widetilde \pi:\widetilde\K_6\to\K_6 \), where \(\widetilde\K_6\) is a projective ihs manifold of dimension 6. 
    \item \cite{lehnsorger}\cite{PR:OG10} Let \(v\) be a Mukai vector such that \(v=2w\) with \(w=(r,l,s)\) primitive and \(l^2-2rs=2\). Then \(\K_v\) is a singular symplectic variety and it admits a symplectic resolution \(\widetilde \pi:\widetilde\K_v\rightarrow \K_v\), where \(\widetilde\K_v\) is a projective ihs manifold deformation equivalent to \(\widetilde\K_6\).
    \item \cite[Theorem 1.17]{PR18} Let \(v_1\) and \(v_2\) be Mukai vectors as in (2). Then \(\K_{v_1}\) and \(\K_{v_2}\) are locally trivial deformation one of the other.
\end{enumerate}
\end{theorem}

\begin{definition}\label{def OG6 K_v}
A moduli space \(\K_v\) as in the hypotheses of \autoref{thm.K_v}.(2) is called a \textit{\(\OG_6^s\) moduli space}.  A primitive symplectic variety that is a locally trivial deformation of \(\K_6\) is called of \textit{\(\OG_6^s\)-type}.
\end{definition}

As a consequence of \autoref{prop primitive sympl} the moduli spaces \(\K_v\) in the theorem above are Namikawa symplectic varieties. 

A desingolarized moduli space \(\widetilde\K_v\) as in \autoref{thm.K_v}.(2) is called a \(\OG_6\) moduli space, and it is by definition an \(\OG_6\)-type ihs manifold.


\subsection{Lattice theory on primitive symplectic varieties} A \textit{lattice}  is a free \(\mathbb{Z}\)-module \(L\) of finite rank endowed with a non-degenerate symmetric bilinear form 
\[ 
q_L\colon L \times L \to \mathbb{Z}.
\] 
A lattice is called \textit{even} if \(q_L(x):=q_L(x,x)\)  is even for any \(x\in L\). Associated to any lattice \(L\) there is a finite abelian group called the \textit{discriminant group} and defined as \[\bA_L=L^{\vee}/L\] where \(L \hookrightarrow L^{\vee}=\Hom_\ZZ(L,\ZZ)\) is the canonical embedding. Observe also that 
\begin{equation}\label{eq.iso.L*}
L^{\vee}\cong\{v\in L\otimes_\ZZ \QQ\ |\ (v,w)\in \ZZ,\ \forall w\in L\}. 
\end{equation}
Given \(v\in L\) we denote by \([v]\) its class in \(\bA_L\) and by \(\divi(v)\) the divisibility of \(v\), i.e. the positive integer generator of the ideal
\((v,L)\subseteq\mathbb{Z}\). Given generators \(\{v_i\}_{i=1}^{\rk(L)}\) of \(L\), the classes \(\{[\frac{v_i}{\divi(v_i)}]\}_{i=1}^{\rk(L)}\) generate \(\bA_L\). 

The order of the discriminant group \(\bA_L\) equals \(|\det(q_L)|\) and this number is called \textit{discriminant} of the lattice \(L\). A lattice \(L\) is called \textit{unimodular} if \(\bA_L=\{\id\}\). 

Because of the isomorphism \eqref{eq.iso.L*}, the pairing on \(L\) induces by \(\QQ\)-linear extension a pairing with \(\QQ\)-values on \(L^{\vee}\) and hence a pairing on \(\bA_L\) with values in \(\QQ/\ZZ\). If the lattice \(L\) is even, then the \(\QQ\)-valued quadratic form on \(L^{\vee}\) gives rise to a quadratic form 
\[
q_{\bA_L}\colon \bA_L \to \QQ/2\ZZ,
\]
called the \textit{discriminant quadratic form} of \(L\).
 
The torsion-free part of second integral cohomology of a primitive symplectic variety has a lattice structure, as we recall in the following theorem.

\begin{theorem}\label{thm BBF and Fujiki} Let \((X,\sigma)\) be a primitive symplectic variety of dimension \(2n\).
\begin{enumerate}
    \item \cite{beauville.varietes}\footnote{This is the reference for the same result in the smooth setting, i.e. the case of ihs manofods.}\cite{kirschner2015}\cite{matsushita15}\cite{Schwald_Fujiki}\cite{BL18} There is a quadratic form \(q_X\) on \(\HH^2(X,\ZZ)\), called \textnormal{Beauville-Bogomolov-Fujiki} (in short, \textnormal{BBF}) form, which is defined up to scaling by \footnote{For a complex compact variety \(X\) the integration over its singular top cohomology group \(\HH^{2n}(X,\ZZ)\) is defined as \(\int_X\alpha:=[X]\cap\alpha\), where \([X]\in \HH_{2n}(X,\ZZ)\) is the fundamental class. Integration commutes with pullbacks by bimeromorphic maps.}
    \[
    q_X(\alpha):=\frac{n}{2}\int_X (\sigma\overline\sigma)^{n-1}\alpha^2+(1-n)\int_X\sigma^n\overline\sigma^{n-1}\alpha\int_X\sigma^{n-1}\overline\sigma^n\alpha 
    \]
    and that is non-degenerate of signature \((3,b_2(X)-3)\). Such form gives a lattice structure to \(\HH^2(X,\ZZ)_{tf}\).
    \item \cite[Theorem 4.7]{fujiki}\footnote{This is the reference for the same result in the smooth setting, i.e. the case of ihs manifods.}\cite[Theorem 2]{Schwald_Fujiki}\footnote{Note that the terminology used in \cite{Schwald_Fujiki} is different from the one we have chosen, as he refers to primitive symplectic varieties as irreducible symplectic varieties, cfr. Definition 1.(3) therein. },\cite[Theorem 5.20]{BL18}. There is a positive real number \(c_X\), called \textnormal{Fujiki constant}, such that for any \(\alpha\in \HH^2(X,\mathbb C)\) the following relation, called \textnormal{Fujiki relation}, holds true:
    \[
    \int_X \alpha^{2n}=c_X\cdot q_X(\alpha)^n.
    \]
\end{enumerate}
\end{theorem}

As consequence of item (2) in \autoref{thm BBF and Fujiki} the BBF form \(q_X\) and the Fujiki constant \(c_X\) are deformation invariant. The definition of the BBF quadratic form and its compatibility with the Hodge structure on \(\HH^2(X,\ZZ)\) are the starting point for Torelli theorems on primitive symplectic varities proved in \cite{BL18}.

\begin{remark}\label{rem BBF and Fujiki invariant up to resolution}
Let \(X\) be a primitive symplectic variety and \(f:\widetilde X\to X\) a symplectic resolution such that \(\widetilde X\) is an ihs manifold. A natural question is whether the inclusion  \(f^*:\HH^2(X,\ZZ)\hookrightarrow\HH^2(\widetilde X,\ZZ)\) in \autoref{prop cohomology sympl resol} is compatible with the BBF-lattice structures, i.e. if
\(q_X(\alpha)=q_{\widetilde X}(f^*\alpha)\) for any \(\alpha\in \HH^2(X,\ZZ)\). The answer to this question is positive  \cite[Corollary 24]{Schwald_Fujiki}. In fact this approach was the first one to be used to define a lattice structure on \(\HH^2(X,\ZZ)\) in the particular case of primitive symplectic varieties admitting a symplectic resolution given by an ihs manifold, see \cite[Theorem 8.2]{namikawa01} and \cite{PR:OG10} in the case of moduli spaces of sheaves. \\
As a consequence, under the same hypotheses one concludes \(c_X=c_{\widetilde{X}}\), i.e.\ the Fujiki constant does not change under symplectic resolutions.
\end{remark}

\begin{example}\label{ex lattice OG6}
When \(X\) is a primitive symplectic variety of \(\OG_6^{s}\)-type then \(\HH^2(X,\ZZ)\) is torsion-free by \autoref{prop cohomology sympl resol} and it is an even lattice of signature \((3,4)\). More precisely, \cite[Theorem 1.7]{PR:OG10} gives 
\[ \HH^{2}(X,\mathbb{Z})\cong \bU^{\oplus3} \oplus [-2] = \colon \bL_{\OG_6^s} \] 
where \(\bU\) is the hyperbolic lattice and \([-2]\) is a rank-one lattice generated by an element of square \(-2\); we call \(\sigma\) this element of square \(-2\). Let \(\bA_{\OG_6^s}\) be the discriminant group of \(\bL_{\OG_6^s}\). Observe that \(\sigma\) has divisibility \(2\) in \(\bL_{\OG_6^s}\) and \(\bU^{\oplus 3}\) is unimodular, hence  \([\frac{\sigma}{\divi(\sigma)}]\) generates the discriminant group \(\bA_{\OG_6^s}\cong \ZZ/2\ZZ\). With this notation it holds that

\begin{equation}\label{eq discriminant form}
q_{\bA_{\OG_6^s}}\bigl(\Bigl[\frac{\sigma}{\divi(\sigma)}\Bigl]\bigl)= \Bigl[\frac{3}{2}\Bigl]\in\QQ/2\ZZ.
\end{equation}
\end{example}


\section{Positive uniruled divisors on O'Grady's moduli spaces}\label{section uniruled on K}

Let \(\K_v\) be an \(\OG_6^s\) moduli space, which is a Namikawa symplectic variety of \(\OG_6^s\)-type. The goal of this section is to prove \autoref{thm_positive}, i.e. the existence of uniruled divisors in any linear system defined by a positive effective divisor on \(\K_v\), up to a positive multiple. 

\subsection{Mongardi-Rapagnetta-Saccà model}\label{subsection MRS model} We consider here a construction firstly introduced by Rapagnetta \cite{rapagnettaOG6} in the case of Mukai vector \((0,2\theta,-2)\), and then generalized by Mongardi-Rapagnetta-Saccà \cite{mong.sac.rapagn} for any Mukai vector as in the assumptions of \autoref{thm.K_v}.(2), i.e. the ones giving an \(\OG_6^s\) moduli space. We briefly recall here the main steps of this construction; the interested reader can find any further detail in the original work \cite{mong.sac.rapagn}.

Let \(\widetilde\K_v\) be the symplectic resolution of \(\K_v\) as in \autoref{thm.K_v}, that is an ihs manifold of \(\OG_6\)-type. Lehn and Sorger \cite{lehnsorger} prove that \(\widetilde \K_v\cong \Bl_{\Sigma_v}\K_v\), where \(\Sigma_v\) is the singular locus of \(\K_v\); we call \(\widetilde \Sigma_v\) the exceptional divisor of this blow-up. Rapagnetta \cite[Theorem 3.3.1]{rapagnettaOG6} proves that \(\widetilde\Sigma_v\) is divisible by two in \(\Pic(\widetilde \K_v)\), hence (being \(\Pic(\widetilde\K_v)\) torsion free) there exists a unique normal projective variety \(\widetilde Y_v\) with a double cover \(\widetilde\varepsilon_v\colon\widetilde Y_v\to\widetilde\K_v\) branched over \(\widetilde\Sigma_v\). 

By \cite[Theorem 4.2]{mong.sac.rapagn} this construction transfers to the variety \(\K_v\), i.e. there exists a normal projective variety \(Y_v\) with a finite 2:1 morphism \(\varepsilon_v\colon Y_v\to \K_v\) with branch locus \(\Sigma_v\); the singular locus of \(Y_v\) is denoted by \(\Gamma_v\) and it consists of 256 points. Finally, we consider the blow-ups \(\overline Y_v:=\Bl_{\Gamma_v} Y_v\) and \(\overline\K_v:=\Bl_{\Omega_v}\K_v\), where \(\Omega_v\subset\K_v\) is the singular locus of \(\Sigma_v\); we call \(\overline \Gamma_v\subset \overline Y_v\) the exceptional locus of the first blow-up. By \cite[Corollary 4.3]{mong.sac.rapagn} there exists a finite 2:1 morphism \(\overline\varepsilon_v\colon\overline Y_v\to\overline \K_v\) branched over \(\overline\Sigma_v:=\Bl_{\Omega_v}\Sigma_v\) lifting \(\varepsilon_v\colon Y_v\to\K_v\), i.e. forming a commutative diagram together with the blow-up morphisms defining \(\overline Y_v\) and \(\overline\K_v\).

The variety \(\overline Y_v\) is smooth and \(Y_v\) is normal. It follows that the morphisms \(\varepsilon_v\) and \(\overline\varepsilon_v\)  induce regular involutions \(\tau_v:Y_v\to Y_v\) and \(\overline\tau_v:\overline Y_v\to\overline Y_v\)  respectively; as \(\K_v\) and \(\overline \K_v\) are normal, the morphisms \(\varepsilon_v\) and \(\overline \varepsilon_v\) are indeed the quotient maps with respect to the involutions \(\tau_v\) and \(\overline\tau_v\), see \cite[Remark 4.7]{mong.sac.rapagn}. We summarize the picture we obtained so far in the following commutative diagram:
\begin{center}
\begin{tikzcd}
    \overline Y_v=\Bl_{\Gamma_v} Y_v \arrow[r, "\overline\varepsilon_v"] \arrow[loop left, "\overline\tau_v"] \arrow[d] & \overline\K_v=\Bl_{\Omega_v}\K_v \arrow[d] \\
    Y_v \arrow[r, "\varepsilon_v"] \arrow[loop left, "\tau_v"] & \K_v
\end{tikzcd}
\end{center}
The exceptional divisor \(\overline\Gamma_v\subset \overline Y_v\) consists of the union  of 256 varieties \(I_{v,i}\) which are isomorphic to the incidence variety \(\PP(V)\times\PP(V)\), where \(V\) is a complex vector space of dimension 4. Let \(p_{v,i}:I_{v,i}\to \PP(V)\) be one of the two natural projections, which is a \(\PP^2\)-fibration; the normal bundle of \(I_{v,i}\) in \(\overline Y_v\) has degree -1 on the fibers of \(p_{v,i}\), hence by Nakano's contraction theorem there exists a complex manifold \(\underline Y_v\) with a morphism of complex manifolds \(h_v:\overline Y_v\to\underline Y_v\) having exceptional locus \(\overline\Gamma_v\) and such that \(J_v:=h_v(\overline\Gamma_v)\) is the union of 256 copies of \(\PP^3\). Furthermore, \(h_v\) realizes \(\overline Y_v\) as the blow-up of \(\underline Y_v\) along \(J_v\). 

The fundamental point of this last construction is that \(\underline Y_v\) is an ihs manifold of \(\K3^{[3]}\)-type, i.e. an ihs manifold deformation equivalent to the Hilbert of \(n\) points on a \(\K3\) surface \cite[Proposition 5.3]{mong.sac.rapagn}. As the involution \(\overline\tau_v\) on \(\overline Y_v\) sends \(\overline \Gamma_v\) to itself, it descends to a rational involution \(\underline\tau_v:\underline Y_v\dashrightarrow \underline Y_v\), regular on \(\underline Y_v\smallsetminus J_v\). By construction of \(\underline Y_v\) there is a regular birational morphism \(f_v:\underline Y_v\to Y_v\) contracting \(J_v\) to \(\Gamma_v\) and such that the following diagram is commutative:
\begin{center}
    \begin{tikzcd}
        & \overline Y_v \arrow[loop right, "\overline\tau_v"] \arrow[dd] \arrow[dl, "h_v"] & \\
        \underline Y_v \arrow[loop left, "\underline\tau_v", dashed] \arrow[dr, "f_v"] & & \\
        & Y_v \arrow[loop left, "\tau_v"] \arrow[r, "\varepsilon_v"] & \K_v
    \end{tikzcd}
\end{center}
As observed in \cite[Remark 5.4]{mong.sac.rapagn} the regular involution \(\overline\tau_v\) exchanges the two \(\PP^2\)-fibrations on each \(I_{v,i}\), hence \(\underline \tau_v\) can not be extended out of \(\underline Y_v\smallsetminus J_v\), as \(\underline Y_v\) is obtained from \(\overline Y_v\) by contracting only one of theme; therefore the indeterminacy locus  of \(\underline\tau_v\) is exactly \(J_v\). 

We call 
\begin{equation}\label{eq.Phi}
    \Phi_v:=\varepsilon_v\circ f_v:\underline Y_v\to \K_v
\end{equation}
which is a generically 2:1 regular morphism contracting \(J_v\subset \underline Y_v\) to points. The \(\OG_6^s\)-type Namikawa symplectic variety \(\K_v\) is in this sense realized as a ‘quotient' of the \(\K3^{[3]}\)-type ihs manifold \(\underline Y_v\) via the rational involution \(\underline\tau_v\).

We conclude relating the BBF-lattice structure of \(\K_v\) and \(\underline Y_v\) via the morphism \(\Phi_v\).

\begin{proposition}\label{prop Phi* on the lattices}
For any \( \alpha\in \HH^2(\K_v,\ZZ)\) we have 
\[q_{\underline Y_v}(\Phi_v^*(\alpha))=2 \cdot q_{\K_v}(\alpha)\]
where \(q_{\underline Y_v}\) and \(q_{\K_v}\) are the BBF-forms on \(\underline Y_v\) and \(\K_v\) respectively.
\end{proposition}
\begin{proof}
We have
\begin{align*}
    c_{\underline Y_v}\cdot q_{\underline Y_v}(\Phi_v^*(\alpha))^3&=\Phi_v^*(\alpha)^6  \hspace{1.9cm}   \textnormal{by \autoref{thm BBF and Fujiki}.(2) on } \underline Y_v \\
    &= 2\cdot \alpha^6 \hspace{2.2cm} \textnormal{as } \Phi_v \textnormal{ is a generically 2:1 finite morphism}  \\
    &=2\cdot c_{\K_v}\cdot q_{\K_v}(\alpha)^3 \hspace{0.5cm} \textnormal{by \autoref{thm BBF and Fujiki}.(2) on } \K_v.
\end{align*}
As observed in \autoref{rem BBF and Fujiki invariant up to resolution}, we have \(c_{\K_v}=c_{\widetilde\K_v}\)\footnote{\cite[Theorem 1.7]{Per_Rap_the_second_integral} computes the Fujiki constant for any singular moduli space on a surface with trivial canonical bundle, but in this particular case, i.e. when the singular moduli space admits a symplectic resolution given by an ihs manifold, the computation is easier as the Fujiki constant equals the one of its resolution.}. As \(c_{\widetilde \K_v}=60\) \cite[Theorem 3.5.1]{rapagnettaOG6} and \(c_{\underline Y_v}=15\) \cite{beauville.varietes} we obtain the statement.
\end{proof}


\subsection{Positive uniruled divisors on \(\K_v\)}\label{subsection positive uniruled on K}
We start recalling the definition of  (polarized) locally trivial deformations of primitive symplectic varieties, that we will use in the main result of this section. We aslo introduce the related notions of moduli space of marked or polarized primitive symplectic varieties, as we are going to use them in the following sections.

\begin{definition}
A \textit{locally trivial family} is a flat proper morphism \(\pi:\mathcal X\to S\)  of complex spaces where \(S\) is connected and such that for any \(p\in \mathcal X\) there exist open neighborhoods \(\mathcal U\subseteq\mathcal X\) of \(p\) and \(S_p\subseteq S\) of \(\pi(p)\) such that 
\[
\mathcal U\cong U\times S_p
\]
over \(S_p\), where \(U:=\mathcal U\cap \pi^{-1}(\pi(p))\). A complex variety \(X_1\) is a  \textit{locally trivial deformation} of another complex varieties \(X_2\) if there exists a locally trivial family \(\pi:\mathcal X\to S\) and points \(s_1,s_2\in S\) such that \(\mathcal X_{s_1}\cong X_1\) and \(\mathcal X_{s_2}\cong X_2\).
\end{definition}

Every small locally trivial deformation of a primitive symplectic variety is a primitive symplectic variety \cite[Corollary 4.11]{BL18}. Let \(\delta\) be a fixed locally trivial deformation type of primitive symplectic varieties; a primitive symplectic variety of this class is called of \(\delta\)-type.

Let \(L\) be a rank \(n\) lattice of signature \((3,n-3)\). A \(L\)-\textit{marked primitive symplectic variety} is a pair \((X,\mu)\) with \(X\) a primitive symplectic variety and \(\mu\) an isometry \(\mu\colon \HH^2(X,\ZZ)_{tf}\to L\), called \(L\)-marking; here \(\HH^2(X,\ZZ)_{tf}\) has the BBF-lattice structure defined in \autoref{thm BBF and Fujiki}. An isometry of \(L\)-marked primitive symplectic varieties \((X,\mu)\) and \((X',\mu')\) is an isomoprhism \(f:X\to X'\) such that \(\mu'=\mu\circ f^*\), where \(f^*\) is the pullback on the second cohomology group.

Let \(\delta\) be a fixed locally trivial deformation type of primitive symplectic varieties. By \autoref{thm BBF and Fujiki} all primitive symplectic varieties of \(\delta\)-type have (the torsion-free part of) their integral second cohomology group isometric to the same abstract lattice \(L_\delta\).

\begin{definition}\label{def mod space of marked psv}
Let \(\mathfrak M_{L_\delta}\) be the \textit{moduli space of \(L_\delta\)-marked primitive symplectic varieties of \(\delta\)-type}, i.e. its elements are \(L_\delta\)-marked primitive symplectic varieties \((X,\mu)\) where \(X\) is of \(\delta\)-type.
\end{definition}
The moduli space \(\mathfrak M_{L_\delta}\) has a (non-necessarily Hausdorff) topology and a complex structure obtained by gluing together the Kuranishi spaces \(\Def(X)^{lf}\) of locally trivial deformations of a primitive symplectic variety of \(\delta\)-type \(X\), which are smooth of dimension \(h^{1,1}(X)\) \cite[Theorem 4.7]{BL18}; the crucial point in the gluing is that miniversal locally trivial deformations are in fact universal \cite[Lemma 4.9]{BL18}. It follows that \(\mathfrak M_{L_\delta}\) is a smooth manifold.

A polarized variety is a pair \((X,h)\) where \(X\) is a projective variety and \(h\) is a primitive polarization, i.e. the first Chern class of a primitive ample line bundle on \(X\). 

\begin{definition}
A \textit{polarized locally trivial family} is a locally trivial family \(\pi:\mathcal X\to S\) together with a line bundle \(\mathcal L\in\Pic(\mathcal X)\). A polarized variety \((X_1,h_1)\) is a \textit{polarized locally trivial deformation} of another polarized variety \((X_2,h_2)\) if there exists a polarized locally trivial family \(\pi:\mathcal X\to S\), \(\mathcal L\in \Pic(\mathcal X)\) and points \(s_1,s_2\in S\) such that \((\mathcal X_{s_1},c_1(\mathcal L_{s_1}))\cong (X_1,h_2)\) and \((\mathcal X_{s_2},c_1(\mathcal L_{s_2}))\cong (X_2,h_2)\) as polarized pairs.
\end{definition}

A polarized primitive symplectic variety is a polarized variety \((X,h)\) such that \(X\) is primitive symplectic.

\begin{definition}\label{def mod space of polarized psv}
For any \(d\in \NN\) let \(\mathfrak M_{\delta, d}\) be the coarse \textit{moduli space of polarized primitive symplectic varities of \(\delta\)-type and degree \(d\)}, i.e. its elements are polarized primitive symplectic varieties \((X,h)\) where \(X\) is of \(\delta\)-type and the BBF-square of \(h\) equals \(d\).
\end{definition}

By \cite[Proposition 8.7 and Lemma 8.8]{BL18} the space \(\mathfrak M_{\delta, d}\) is a quasi-projective scheme. Finally, we call 
\[
\mathfrak M_\delta:=\amalg_{d\in\NN}\mathfrak M_{\delta, d}.
\]
Observe that two polarized primitive symplectic varieties of \(\delta\)-type that are one deformation of the other (as polarized varieties) belong to the same component \(\mathfrak M_{\delta, d}\) of \(\mathfrak M_\delta\), but the viceversa is not true in general (see \autoref{subsection monodromy_OG6_eichler}). In what follows we will be interested in the case \(\delta=\OG_6^s\), i.e. to the moduli space \(\mathfrak M_{\OG_6^s}\).

We observe here a fact we are going to use frequently in what follows. 
\begin{remark}\label{rem_pos2ample}
Let \(X\) be a primitive symplectic variety  and \(h\) an effective class that is positive with respect to the BBF-form of \(X\). Consider a small locally trivial deformation \((X',h')\) of the pair \((X,h)\) such that \(\Pic(X')\cong\ZZ\) and \(h'\) is still effective; as the BBF-form is deformation invariant, the class \(h'\) is positive on \(X'\). By the projectivity criterion \cite[Theorem 3.11]{huybrechts}\footnote{This is the reference in the smooth setting, i.e. for ihs manifolds.}\cite[Theorem 1.2]{BL18} we conclude that \(X'\) is projective, hence \(h'\) is ample. We conclude that a pair \((X,h)\) as above always deforms to a pair \((X',h')\) with \(h'\) ample class on the primitive symplectic variety \(X'\). Observe that if the starting variety \(X\) is an ihs manifold then the variety \(X'\) is also ihs manifold, as it is a deformation of the first one \cite{beauville.varietes}.
\end{remark}

\begin{definition}\label{def uniruled}
Let \(D\subseteq X\) be an irreducible divisor. We say that \(D\) is uniruled if there exists a \(\dim(D)-1\) variety \(Y\) and a finite dominant morphism \(\PP^1\times Y\dashrightarrow D\).
\end{definition}

We are finally ready to prove the main result of this section.

\begin{theorem}\label{thm_positive}
Let \(\K_v\) be an \(\OG_6^s\) moduli space, with BBF-form \(q_{\K_v}\). For any effective and \(q_{\K_v}\)-positive\footnote{A divisor is positive if so is its first Chern class.} divisor \(D\) on \(\K_v\) there exists a positive integer \(m\) such that the linear system \(|mD|\) contains an irreducible uniruled divisor. 
\end{theorem}
\proof
Consider the generically \(2 : 1\) regular morphism defined in \eqref{eq.Phi}:
\[
\Phi_v\colon\underline Y_v \to \K_v
\]
where \(\underline Y_v\) is an ihs manifold of \(\K3^{[3]}\)-type. As \(c_1(D)\in \HH^2(\K_v,\ZZ)\) is by assumption a \(q_{\K_v}\)-positive class, the pull-back \(\Phi_v^*c_1(D)\in \HH^2(\underline Y_v,\ZZ)\) is  \(q_{\underline Y_v}\)-positive by \autoref{prop Phi* on the lattices}. By Remark \ref{rem_pos2ample} there exists a deformation \((Y'_v,c_1(D)')\) of \((\underline Y_v,\Phi_v^*c_1(D))\) such that \(D'\) is an ample divisor. Using \cite[Theorem 1.1 and Remark 1.3(iii)]{CPM} on the polarized \(\K3^{[3]}\)-type ihs manifold \((Y'_v,c_1(D'))\) we deduce that there exists a positive integer \(k\) such that the linear system \(|kD'|\) contains an irreducible uniruled divisor. Using the deformation result \cite[Corollary 3.5]{CPM} we conclude that a also multiple of the original class \(\Phi_v^*c_1(D)\) has an irreducible representative \(D''\) that is uniruled. Then the divisor \(\Phi_v(D'')\) on \( \K_v\), taken with reduced schematic structure, is irreducible and uniruled, and its class in \(\HH^2(\K_v,\ZZ)\) is a multiple of \(c_1(D)\).
\endproof


\section{Monodromy orbits}\label{section.oribts}

In this section we prove our main result \autoref{main thm}, which states that any point of \(\mathfrak M_{\OG_6^s}\) corresponds to a polarized \(\OG_6^s\)-type primitive symplectic variety whose polarization is proportional to the first Chern class of a uniruled divisor. The first step in this direction is  to characterize the connected components of \(\mathfrak M_{\OG_6^s}\), which are given by the locally-trivial monodromy group.


\subsection{Monodromy of \(\OG_6^{s}\) and Eichler's criterion}\label{subsection monodromy_OG6_eichler}
Let \(X\) be a primitive symplectic variety of \(\OG_6^s\)-type. We denote by \(\Mon^2(X)\) the locally trivial monodromy group, i.e. the group of automorphisms of \(\HH^2(X,\ZZ)\) given by parallel transport operators along locally triavial families (see \cite[Definition 2.11]{mongardi.rapagn.monOG6}). In holds that \(\Mon^2(X)\subseteq O^+(\HH^2(X,\ZZ))\) \cite[Lemma 2.12]{mongardi.rapagn.monOG6}, where \(O^+(\HH^2(X,\ZZ))\) is the group of orientation preserving isometries, which are the isometries preserving the orientation of the positive cone in \(\HH^2(X,\ZZ)\). The monodromy group is invariant up to locally trivial deformations, and we denote  by \(\Mon^2(\bL_{\OG_6^s})\) the monodromy group acting on \(\bL_{\OG_6^s}\) (see \autoref{ex lattice OG6} for the definition of \(\bL_{\OG_6^s}\)), defined via a \(\bL_{\OG_6^s}\)-marking \(\mu:\HH^2(X,\ZZ)\to\bL_{\OG_6^s}\). 

Connected components of the moduli space \(\mathfrak M_{\OG_6^s}\) are given by orbits under the action of the locally trivial monodromy group. In other words, two elements in \(\mathfrak M_{\OG_6}^s\) are in the same connected component if their polarizations, seen in \(\bL_{\OG_6^s}\), are in the same \(\Mon^2(\bL_{\OG_6^s})\)-orbit. 

The locally trivial monodromy group of primitive symplectic varieties of \(\OG_6^s\)-type is known, and as in the case of \(\OG_6\)-type ihs manifolds it turns out to be maximal.

\begin{proposition}[{\cite[Proposition 4.12]{mongardi.rapagn.monOG6}}]\label{prop Mon^2(OG6)} Let \(X\) be a primitive symplectic variety of \(\OG_6^{s}\)-type, then
\[
\Mon^{2}(X)=O^{+}(\HH^{2}(X,\mathbb{Z})).
\]
\end{proposition}

Thanks to the previous result, the following criterion gather a direct way to compute the \(\Mon^2(\bL_{\OG_6^s})\)-orbits of elements in the lattice \(\bL_{\OG_6^s}\).

\begin{proposition}[Eichler's criterion, {\cite[Proposition 3.3(i)]{GHS_abel}}]\label{prop Eichler}
Let \(L\) be an even lattice that contains two orthogonal copies of the hyperbolic plane, and assume that the discrimitant \(\bA_L\) is a cyclic group. If two vectors \(u\) and \(v\) in \(L\) have the same square and divisibility in \(L\), then there exist a transvection\footnote{For the definition of transvection, or Eichler transvection, we refer to \cite[Lemma 3.1]{GHS_abel}} \(\tau\) such that \(\tau(u)=v\).
\end{proposition}

\begin{corollary}\label{cor monodromy}
Connected components of the moduli space \(\mathfrak M_{\OG_6^s}\) are determined by the degree and the divisibility of the polarization of their elements.
\end{corollary}
\begin{proof}
Every transvection of \(\bL_{\OG_6^s}\) is an orientation preserving  isometry \cite[\S 3.1]{GHS_abel}, hence a monodromy operator by \autoref{prop Mon^2(OG6)}. The lattice \(\bL_{\OG_6^s}\) contains two orthogonal copies of the hyperbolic plane and it has cyclic discriminant group (see \autoref{ex lattice OG6}), hence we can appy \autoref{prop Eichler} obtaining that the orbit of a vector in \( \bL_{\OG_6^s}\) with respect to the monodromy group is determined by its degree and divisibility. 
\end{proof}

\subsection{Main results}\label{subsection main results} We start analyzing the possible combinations of degree and divisibility appearing in the lattice \(\bL_{\OG_6^s}\) introduced in \autoref{ex lattice OG6}. The goal of the first part of this section is to exhibit representatives for any possible combination of divisibility and positive square in \(\bL_{\OG_6^s}\), constructed as the class of an (up to multiple) effective divisor on some \(\OG_6^s\) moduli space \(\K_v\), see \autoref{prop invariants of classes on Kv}. 

\begin{lemma}\label{lemma_ type_d}
Let \((X,h)\) be a polarized primitive symplectic variety of \(\OG_6^{s}\) type, with BBF-form \(q_X\). We call \(q_X(h)=2d\); if \(\divi(h)=2\) in \(\HH^2(X,\ZZ)\) then \(d\equiv 3 \mod 4\). 
\end{lemma}
\proof
Fix a marking \(\eta\colon \HH^{2}(X,\mathbb{Z}) \to \bL_{\OG_6^{s}}\); by abuse of notations we call \(h\) the class \(\eta(h)\in \bL_{\OG_6^s}\). Throughout the proof we assume \(\divi(h)=2\). Consider the class \([\frac{h}{\divi(h)}]\in \bA_{\OG_6^s}\). By \eqref{eq discriminant form}, the possible values of the induced quadratic form \(q_{\bA_{\OG_6^s}}\) on \(\bA_{\OG_6^s}\) are \(\{0,\frac{3}{2}\} \subset \frac{\QQ}{2\ZZ}\)  and the only class of square \(0\) is the zero class of the discriminant group.  Assume that \(d\equiv 0 \mod  4 \);  it follows that 
\[
q_{\bA_{\OG_6^s}}\Bigl(\Bigl[\frac{h}{\divi(h)}\Bigl]\Bigl)=\frac{d}{2}=0 \in \frac{\QQ}{2\ZZ}
\] 
 hence  \([\frac{h}{\divi(h)}]\) is the zero class of \(\bA_{\OG_6^s}\); we conclude that \(\frac{h}{\divi(h)}\) is an element of the lattice \( \bL_{\OG_6}\) hence \(\divi(h)=1\), which contradicts our assumption. If \(d\equiv 1\) or \(2 \mod  4 \) then the same computation shows that the class \([\frac{h}{\divi(h)}]\) has \(q_{\bA_{\OG_6^s}}\)-square \(\frac{1}{2}\) or \(1\in \frac{\QQ}{2\ZZ}\) respectively,  and we have already noticed that this case does not occur.
\endproof

\begin{proposition}\label{prop invariants of classes on Kv} Let \((X,h)\) be a polarized primitive symplectic variety of \(\OG^s_6\)-type, with BBF-form \(q_X\); we call \(q_X(h)=2d\) and \(\divi(h)=e\) in \(\HH^2(X,\ZZ)\). Then there exists a \(\OG_6^s\) moduli space \(\K_v\) and a class \(\alpha\in \HH^2(\K_v,\ZZ)\) such that \(q_{\K_v}(\alpha)=2d\) and \(\divi(\alpha)=e\) in \(\HH^2(\K_v,\ZZ)\), where \(q_{\K_v}\) is the BBF-form on \(\K_v\). Furthermore, \(\alpha\) or \(2\alpha\) is an effective class.
\end{proposition}
\begin{proof}
Assume that \(\divi(h)=e=1\). Consider an abelian surface \(A\) of degree \(2d\), i.e. having a polarization \(h_A\) of degree \(2d\). The image of \(h_A\) via the primitive inclusion \(\HH^2(A,\ZZ)\hookrightarrow \HH^2(\K_{(2,0,-2)}(A), \ZZ)\) \cite[Theorem 3.5.1]{rapagnettaOG6} gives an effective divisor of square \(2d\) and divisibility \(1\) on the \(\OG_6^s\) moduli space \(\K_{(2,0,-2)}(A)\).

For the case \(\divi(h)=e=2\), we fix an abelian surface \(A\) of degree 2, i.e. having a polarization \(h_A\) of degree 2. By abuse of notation we call \(h_A\) the image of it via the inclusion \(\mu:\HH^2(A,\ZZ)\hookrightarrow \HH^2(\K_{(2,0,-2)}(A), \ZZ)\). Let \(B\subseteq\K_{(2,0,-2)}(A)\) be the Weil effective divisor parametrizing non-locally free sheaves and call \(b:=\frac{1}{2}c_1(2B)\). As consequence of \cite[Theorem 2.4 and Theorem 2.5]{PR18} we have:
\[
\HH^2(\K_{(2,0,-2)}(A),\ZZ)=\mu( \HH^2(A,\ZZ))\oplus_{\perp} \ZZ \cdot b
\]
and the class \(b\) has square -2. Taking \(A\) general we can assume that \(\NS(\K_{(2,0,-2)}(A))=\mu(\NS(A))\oplus_\perp \ZZ \cdot b\) having BBF-matrix 
\[\begin{pmatrix}
h_A^2 & h_A\cdot b \\
h_A\cdot b & b^2
\end{pmatrix}=\begin{pmatrix}
2 & 0 \\
0 & -2
\end{pmatrix}.
\]
For any odd integer \(n\ge 1\) we consider the class
\[
\alpha_{n}:=(n+1)h_A+n b\in\HH^2(\K_{(2,0,-2)},\ZZ) 
\] 
and observe that by construction \(2\alpha_n\) is an effective class. Using the above written decomposition of the lattice \(\HH^2(\K_{(2,0,-2)}(A),\ZZ)\) a straightforward computation gives that \(\alpha_n\) has divisibility 2 in it. Writing \(n=2l+1\) and \(q_{\K_{(2,0,-2)}(A)}(\alpha_n)=2d\) we have 
\[
d=2n+1=4l+3\equiv 3\mod 4
\]
and varying \(l\in \NN\) (hence \(n\ge 1\) odd integer) we obtain all possible positive integers of this form. By \autoref{lemma_ type_d} such \(2d\) are the only possible square for a class of divisibility 2 on an \(\OG^s_6\)-type primitive symplectic variety, hence the statement is proved.
\end{proof}

We are finally ready to prove our main theorem.

\begin{proof}[Proof of \autoref{main thm}]
For any \((X,h)\) as in the statement consider a pair \((\K_v,\alpha)\) given by \autoref{prop invariants of classes on Kv}, i.e. where \(\K_v\) is a \(\OG_6^s\) moduli space and \(\alpha\) is a effective class on \(\K_v\), having the same square and divisibility of \(h\). Furthermore, \(\varepsilon\alpha\) is an effective class, where \(\varepsilon=1\) if \(\divi(\alpha)=1\) and \(\varepsilon=2\) if \(\divi(\alpha)=2\). By \autoref{rem_pos2ample}, there exists a polarized locally trivial deformation \((X',h')\)  of  \((\K_v,\varepsilon\alpha)\) such that \(h'\) is an ample class, with the same square and divisibility of \(\varepsilon\alpha\). In conclusion, we have obtained a polarized primitive symplectic variety \((X',h')\) having the same numerical invariants of \((X,\varepsilon h)\) and that is a polarized locally trivial deformation of \((\K_v,\varepsilon\alpha)\).

By \autoref{cor monodromy} we have that \((X,\varepsilon h)\) and \((X',h')\) are in the same connected component of the moduli space \(\mathfrak M_{\OG_6^s}\), hence we obtain one from the other with a polarized locally trivial deformation; we conclude that  \((X,\varepsilon h)\)  is also a polarized locally trivial deformation of  \((\K_v,\varepsilon\alpha)\). By \autoref{thm_positive} there exists a positive integer \(m'\) such that the class \(m'\varepsilon\alpha\) is the first Chern class of an irreducible uniruled divisor. We can finally apply \cite[Theorem 1.1]{LMP}, concluding that there exists a positive integer \(k\) such that also the class \(km' \varepsilon h\) is the first Chern class of a uniruled divisors, hence we have the statement for \(m=km'\varepsilon\).
\end{proof}

\subsection{The smooth case}\label{section smooth case}
It is natural to ask if our methods to prove \autoref{main thm} can be applied to obtain results about the existence of ample uniruled divisors on \textit{smooth} \(\OG_6\)-type ihs manifold. Also the smooth setting the monodromy group is maximal \cite[Theorem 5.4]{mongardi.rapagn.monOG6}, hence the same result as in \autoref{cor monodromy} holds true for the moduli space of polarized ihs manifolds of \(\OG_6\)-type. Following the proof of \autoref{main thm}, what one needs to show is the existence of uniruled representatives in each positive effective class on a \(\OG_6\) moduli space \(\widetilde \K_v\) in analogy with \autoref{thm_positive}, combined with the fact that positive and (up to multiple) effective classes on \(\widetilde \K_v\) cover all possible squares and divisibilities of polarizations on a \(\OG_6\)-type ihs manifold, as in \autoref{prop invariants of classes on Kv} for the singular case. This last point is indeed true, as we show in the following.

\begin{proposition}
Let \((X,h)\) be a polarized ihs manifold of \(\OG_6\)-type, with BBF-form \(q_x\); we call \(q_X(h)=2d\) and \(\divi(h)=e\) in \(\HH^2(X,\ZZ)\). Then there exists a \(\OG_6\) moduli space \(\widetilde \K_v\) and a  class \(\alpha\in \HH^2(\widetilde\K_v, \ZZ)\) such that \(q_{\widetilde\K_v}(\alpha)=2d\) and \(\divi(\alpha)=e\) in \(\HH^2(\widetilde\K_v, \ZZ)\), where \(q_{\widetilde\K_v}\) is the BBF-form of \(\widetilde\K_v\). Furthermore, \(\alpha\) or \(2\alpha\) is an effective class.
\end{proposition}
\begin{proof}
The proof follows the same steps as the proof of \autoref{prop invariants of classes on Kv}. As first, we prove an analogue of \autoref{lemma_ type_d}. By \cite[Theorem 3.5.1]{rapagnettaOG6} the lattice of an \(\OG_6\)-type ihs manifold is isomorphic to the lattice
\[
\bL_{\OG_6}:=\bU^{\oplus 3}\oplus[-2]\oplus[-2].
\]
We call \(\sigma_1\) and \(\sigma_2\) the generators of square \(-2\) of the rank-one lattices in the decomposition above, and \(\bA_{\OG_6}\) the discriminant group of \(\bL_{\OG_6}\). Hence \(\bA_{\OG_6}\cong \ZZ/2\ZZ\oplus \ZZ/2\ZZ\) is generated by \([\frac{\sigma_1}{\divi(\sigma_1)}]\) and \([\frac{\sigma_2}{\divi(\sigma_2)}]\), and its BBF-form is given in this basis by the matrix
\[
[q_{\bA_{\OG_6}}]=  \begin{pmatrix}
\frac{3}{2} & 0 \\
0 & \frac{3}{2}
\end{pmatrix}\in \M_{2,2}(\QQ/2\ZZ).
\]
Let \(h\in \bL_{\OG_6}\) be an element of divisibility equal to 2, and call \(q_{\bL_{\OG_6}}(h)=2d\). Then \(d\equiv 2,3 \mod 4\): possible squares in \(\bA_{\OG_6}\) are \(\{0,1,\frac{3}{2}\}\subseteq \QQ/2\ZZ\), hence arguing as in the proof of \autoref{lemma_ type_d} we can exclude this time only that \(d\equiv 0,1 \mod 4\), getting our assertion. 

The class \(\alpha\) as in the statement is obtained in the case of \(\divi(\alpha)=1\) or \(\divi(\alpha)=2\) and \(d\equiv 3 \mod 4\) using the inclusion of lattices \(\widetilde\pi^*:\HH^2(\K_v,\ZZ)\hookrightarrow \HH^2(\widetilde\K_v, \ZZ)\) (see \autoref{rem BBF and Fujiki invariant up to resolution}) and then \autoref{prop invariants of classes on Kv}. The only remaining case is the one of divisibility 2 and \(d\equiv 2 \mod 4\), and we proceed as follows. Again by \cite[Theorem 3.5.1]{rapagnettaOG6} we have
\[
\HH^2(\widetilde\K_{(2,0,-2)}, \ZZ)= \mu(\HH^2(A,\ZZ))\oplus_\perp \ZZ\cdot (\widetilde b+a) \oplus_\perp \ZZ\cdot a
\]
where \(\mu\) is an inclusion of lattices, \(a\) is the first Chern class of half of the exceptional divisor of the resolution \(\widetilde\pi:\widetilde\K_{(2,0,-2)}\to\K_{(2,0,-2)}\) and \(\widetilde b\) is the first Chern class of the strict transform via \(\widetilde\pi\) of the divisor of non-locally free sheaves in \(\K_{(2,0,-2)}\). Furthermore, this decomposition gives the isomorphism with the lattice \(\bL_{\OG_6}\), respecting the direct sum decompositions of the lattices as written above.
Let \(A\) be an abelian surface of degree 2 and let \(h_A\in \NS(\widetilde \K_{(2,0,-2)}(A))\) be the image of this polarization via \(\mu\). For any odd integer \(n\ge 1\) we finally consider the class
\[
\beta_n:=(n+1)h_A+n(\widetilde b+a)+a
\]
and observe that \(2\beta_n\) is effective; a straightforward computation gives \(\divi(\beta_n)=2\). Writing \(n=2l+1\) and \(q_{\widetilde\K_{(2,0,-2)}}(\beta_n)=2d\) we have 
\[
d=2n=4l+2\equiv 2 \mod 3
\]
and varying \(l\in\NN\) we obtain all possible positive integer of this form, hence the statement is proved.
\end{proof}

To conclude, we show why our method to prove \autoref{thm_positive} breaks down in the smooth case, going over the steps of its proof in the singular setting. Our starting point has been the regular morphism \(\Phi_v:\underline Y_v\to \K_v\) defined in \eqref{eq.Phi}, that we used to pullback an effective divisor \(D\) on \(\K_v\) to obtain that it has a positive multiple linearly equivalent to an irreducible uniruled divisor. The irreducibility of the divisor is crucial, as the ruling curve needs to be reduced and irreducible to apply the deformation result \cite[Theorem 1.1]{LMP} in the proof of \autoref{main thm}. The uniruled divisor in \(\K_v\) is obtained as  image of a uniruled divisor on \(\underline Y_v\), whose existence is ensured by \cite[Theorem 1.1]{CPM} but whose explicit definition is not given. A natural way to transfer this argument to the smooth setting is to consider the rational map
\[
\widetilde\Phi_v:\underline Y_v\dashrightarrow \widetilde \K_v
\]
given by the composition of \(\Phi_v\) with the inverse of the birational morphism \(\widetilde\pi:\widetilde\K_v\to\K_v\), hence in other words, to pullback to \(\widetilde\K_v\) the uniruled divisor on \(\K_v\) via the resolution \(\widetilde\pi\). The problem of this strategy is that we have no control on the irreducibility of the resulting divisors: a pulled-back divisor could contain \(m\) copies of the (uniruled) exceptional divisor \(\widetilde\Sigma_v\subseteq\widetilde \K_v\), as a uniruled divisor on \(\K_v\) could contain the singular locus \(\Sigma_v\subseteq\K_v\). Since we have no explicit description of the uniruled divisors on \(\underline Y_v\), and hence on \(\K_v\), we can not control whether the uniruled divisors on \(\K_v\) contain \(\Sigma_v\) or not. If we knew the number \(m\) of copies of \(\widetilde\Sigma_v\) in the pullback of a uniruled divisor, we could produce a new irreducible uniruled divisor (the one obtained subtracting \(m\) copies of \(\widetilde\Sigma_v\) from the pullback). Anyway in order to compute the lattice invariants of such divisor we need to know its cohomology class, i.e. the integer \(m\), as we know the lattice invariants of the pulled-back divisors. 

We also remark that in our case we can not use \cite[Proposition 4.5]{LMP}, as it only holds true for singular moduli spaces with terminal singularities, which is not the case for \(\K_v\), see \cite[\S 1.1]{PR18}. 


\bibliography{Biblio}

\begin{thebibliography}{CMP19}

\bibitem[Bea83]{beauville.varietes}
Arnaud Beauville.
\newblock Vari{\'e}t{\'e}s {K}{\"a}hleriennes dont la première classe de
  {C}hern est nulle.
\newblock {\em Journal of Differential Geometry}, 18(4):755--782, 1983.

\bibitem[Bea00]{beauville.symplectic}
Arnaud Beauville.
\newblock Symplectic singularities.
\newblock {\em Invent. Math.}, 139(3):541--549, 2000.

\bibitem[Ber21]{bertini2021rational}
Valeria Bertini.
\newblock Rational curves on o'grady's tenfolds.
\newblock {\em arXiv preprint arXiv:2101.01029}, 2021.

\bibitem[BGL22]{BGL22}
Benjamin Bakker, Henri Guenancia, and Christian Lehn.
\newblock Algebraic approximation and the decomposition theorem for k{\"a}hler
  calabi--yau varieties.
\newblock {\em Inventiones mathematicae}, pages 1--54, 2022.

\bibitem[BL18]{BL18}
Benjamin Bakker and Christian Lehn.
\newblock The global moduli theory of symplectic varieties.
\newblock {\em arXiv preprint arXiv:1812.09748}, 2018.

\bibitem[CMP19]{CPM}
Fran{\c{c}}ois Charles, Giovanni Mongardi, and Gianluca Pacienza.
\newblock Families of rational curves on holomorphic symplectic varieties and
  applications to zero-cycles.
\newblock {\em arXiv preprint arXiv:1907.10970, to appear in Compositio Math.},
  2019.

\bibitem[Fuj87]{fujiki}
Akira Fujiki.
\newblock On the de {R}ham cohomology group of a compact {K}{\"a}hler
  symplectic manifold.
\newblock In {\em Algebraic Geometry, Sendai, 1985}, pages 105--165.
  Mathematical Society of Japan, 1987.

\bibitem[GHS09]{GHS_abel}
Valery Gritsenko, Klaus Hulek, and Gregory Sankaran.
\newblock Abelianisation of orthogonal groups and the fundamental group of
  modular varieties.
\newblock {\em J. Algebra}, 322(2):463--478, 2009.

\bibitem[GKP16]{greb2016singular}
Daniel Greb, Stefan Kebekus, and Thomas Peternell.
\newblock Singular spaces with trivial canonical class.
\newblock In {\em Minimal models and extremal rays (Kyoto, 2011)}, pages
  67--113. Mathematical Society of Japan, 2016.

\bibitem[HP19]{horing2019algebraic}
Andreas H{\"o}ring and Thomas Peternell.
\newblock Algebraic integrability of foliations with numerically trivial
  canonical bundle.
\newblock {\em Inventiones mathematicae}, 216(2):395--419, 2019.

\bibitem[Huy99]{huybrechts}
Daniel Huybrechts.
\newblock Compact hyperk{\"a}hler manifolds: basic results.
\newblock {\em Inventiones mathematicae}, 135(1):63--113, 1999.

\bibitem[Kir15]{kirschner2015}
Tim Kirschner.
\newblock {\em Period mappings with applications to symplectic complex spaces},
  volume 2140.
\newblock Springer, 2015.

\bibitem[LMP21]{LMP}
Christian Lehn, Giovanni Mongardi, and Gianluca Pacienza.
\newblock Deformations of rational curves on primitive symplectic varieties and
  applications.
\newblock {\em arXiv preprint arXiv:2103.16356, to appear in Algebraic
  Geometry}, 2021.

\bibitem[LS06]{lehnsorger}
Manfred Lehn and Christoph Sorger.
\newblock La singularit{\'e} de {O}'{G}rady.
\newblock {\em Journal of Algebraic Geometry}, 15, 10 2006.

\bibitem[Mat15]{matsushita15}
Daisuke Matsushita.
\newblock On base manifolds of lagrangian fibrations.
\newblock {\em Science China Mathematics}, 58(3):531--542, 2015.

\bibitem[MP17]{mongardi.pacienza}
Giovanni Mongardi and Gianluca Pacienza.
\newblock Polarized parallel transport and uniruled divisors on deformations of
  generalized {K}ummer varieties.
\newblock {\em International Mathematics Research Notices},
  2018(11):3606--3620, 2017.

\bibitem[MP19]{mongardi2019corrigendum}
Giovanni Mongardi and Gianluca Pacienza.
\newblock Corrigendum and {A}ddendum to “{P}olarized parallel transport and
  uniruled divisors on generalized {K}ummer varieties”.
\newblock {\em International Mathematics Research Notices}, 2019.

\bibitem[MR19]{mongardi.rapagn.monOG6}
Giovanni Mongardi and Antonio Rapagnetta.
\newblock Monodromy and birational geometry of {O}'{G}rady's sixfolds.
\newblock {\em arXiv preprint arXiv:1909.07173}, 2019.

\bibitem[MRS18]{mong.sac.rapagn}
Giovanni Mongardi, Antonio Rapagnetta, and Giulia Saccà.
\newblock The hodge diamond of o’grady’s six-dimensional example.
\newblock {\em Compositio Mathematica}, 154(5):984–1013, 2018.

\bibitem[Nam01]{namikawa01}
Yoshinori Namikawa.
\newblock Extension of 2-forms and symplectic varieties.
\newblock {\em J. Reine Angew. Math.}, 539:123--147, 2001.

\bibitem[O'G99]{OG10}
Kieran O'Grady.
\newblock Desingularized moduli spaces of sheaves on a {K}3.
\newblock {\em J. Reine Angew. Math.}, 512(512):49--117, 1999.

\bibitem[O'G03]{OG6}
Kieran O'Grady.
\newblock A new six dimensional irreducible symplectic variety.
\newblock {\em J. Algebraic Geometry}, 12(3):435--505, 2003.

\bibitem[PR13]{PR:OG10}
Arvid Perego and Antonio Rapagnetta.
\newblock Deformation of the {O}'{G}rady moduli spaces.
\newblock {\em Journal f{\"u}r die reine und angewandte Mathematik (Crelles
  Journal)}, 2013(678):1--34, 2013.

\bibitem[PR18]{PR18}
Arvid Perego and Antonio Rapagnetta.
\newblock The moduli spaces of sheaves on k3 surfaces are irreducible
  symplectic varieties.
\newblock {\em arXiv preprint arXiv:1802.01182}, 2018.

\bibitem[PR20]{Per_Rap_the_second_integral}
Arvid Perego and Antonio Rapagnetta.
\newblock The second integral cohomology of moduli spaces of sheaves on k3 and
  abelian surfaces, 2020.

\bibitem[Rap07]{rapagnettaOG6}
Antonio Rapagnetta.
\newblock Topological invariants of {O}’{G}rady’s six dimensional
  irreducible symplectic variety.
\newblock {\em Mathematische Zeitschrift}, 256(1):1--34, 2007.

\bibitem[Sch20a]{Schwald_Fujiki}
Martin Schwald.
\newblock Fujiki relations and fibrations of irreducible symplectic varieties.
\newblock {\em \'{E}pijournal G\'{e}om. Alg\'{e}brique}, 4:Art. 7, 19, 2020.

\bibitem[Sch20b]{schwald20}
Martin Schwald.
\newblock On the definition of irreducible holomorphic symplectic manifolds and
  their singular analogs.
\newblock {\em arXiv preprint arXiv:2003.06004}, 2020.

\bibitem[Voi16]{voisin.remarks}
Claire Voisin.
\newblock Remarks and questions on coisotropic subvarieties and 0-cycles of
  hyper-{K}{\"a}hler varieties.
\newblock In {\em K3 surfaces and their moduli}, pages 365--399. Springer,
  2016.

\end{thebibliography}
\bibliographystyle{alpha}


\end{document}